\documentclass{mscs}   

\usepackage{spreenmacros}

\begin{document}

\title[Some results related to the continuity problem]{Some results related to the continuity problem}
\author[Dieter Spreen]{
D\ls I\ls E\ls T\ls E\ls R\ns S\ls P\ls R\ls E\ls E\ls N$^1$%
\thanks{The research leading to these results has received funding from the People Programme (Marie Curie Actions) of the European Union's Seventh Framework Programme FP7/2007-2013/ under REA grant agreement no.\ PIRSES-GA-2011-294962-COMPUTAL.}\\
$^1$ Department of Mathematics, University of Siegen, 57068 Siegen, Germany
 and \addressbreak
Department of Decision Sciences, University of South Africa
 PO Box 392, \addressbreak 
 Pretoria 0003, South Africa. \addressbreak
Email: {\tt spreen@math.uni-siegen.de}
}

\date{\today}
\maketitle

\begin{abstract}
The continuity problem, i.e., the question whether effective maps between effectively given topological spaces are effectively continuous, is reconsidered. In earlier work it was shown that this is always the case, if the effective map also has a witness for noninclusion. The extra condition does not have an obvious topological interpretation. As is shown in the present paper,  it appears naturally where in the classical proof that sequentially continuous maps are continuous the Axiom of Choice is used. The question is therefore whether the witness condition appears in the general continuity theorem only for this reason, i.e., whether effective operators are effectively sequentially continuous. For two large classes of spaces covering all important applications it is shown that this is indeed the case. The general question, however, remains open.    

Spaces in this investigation are in general \emph{not} required to be Hausdorff. They only need to satisfy the weaker $T_0$ separation condition.
\end{abstract}

\section{Introduction}

Computations are usually required to end in finite time. Because of this only a \emph{finite} amout of information about the input can be used during a computation. Moreover, an output once written on the output tape cannot be changed anymore: given more information about the input, the machine can only extend what is already written on the output tape (Monotonicity).

These properties not only hold for functions on the natural numbers, but also for the computation of operators on such functions. A natural topology can be defined on such spaces with respect to which computable operators turn out to be (effectively) continuous.

If one restricts one's interest to functions which are computable and can therefore be presented by the programs computing them (or their codings), there is another way of specifying the computability of operators: an operator is \emph{effective} if it is tracked by a computable function on the code.

 The continuity problem is the question whether effective operators are the restrictions (to computable inputs) of (effectively) continuous operators. Obviously, both approaches are rather unconnected. Nevertheless for certain important cases positive solutions were presented: In the case of operators on the partial computable functions this is due to Myhill and Shepherdson (1955); in the case of the total computable functions to Kreisel, Lacombe and Shoenfield (1959). In the first case the result has been generalised to certain types of directed-complete partial orders with the Scott topology (cf.~e.g.\  \cite{ec76,st78,wd80}), in the other to separable metric spaces \cite{ce62,mo64}. These two types of spaces are quite different, not only topologically: they also offer different algorithmic techniques to use. As follows from an example by Friedberg (1958), effective operators are not continuous, in general.

The situation remained unclear for quite a while. Spreen and Young (1984) showed that for second-countable topological $T_0$ spaces effective maps are effectively continuous if they have a witness for noninclusion. The requirement says that if the image of a basic open set under the operator is not included in a given basic open set in its co-domain, then one needs be able to effectively produce a witness for this. Later, in \cite{sp98}, a mathematically more civilized framework for the derivation of this result was developed.

The condition seems natural when dealing with continuity. In the present note we will give even more evidence for its canonicity. In classical topology it is well known that for second-countable spaces sequentially continuous maps are continuous. The proof can be transferred into a constructive framework. There is however one step in which the classical proof uses the Axiom of Choice and the effective information needed here is exactly what is provided by the witness for noninclusion condition.

So, the question comes up whether effective operators are effectively sequentially continuous and the extra condition is only needed for the step from effective sequential continuity to effective continuity. We will show for a large class of spaces that effective operators are effectively sequentially continuous. To this end we require the spaces to come equipped with a set of canonical computable sequences which are such that  sequences can be \emph{stretched} by wait-and-see strategies and the operator taking convergent sequences to their limits is effective. In addition all basic open sets need be completely enumerable, uniformly in their index. All these conditions seem very natural, but as we will see, in particular the combination of wait-and-see strategies with the computation of limits has a strong impact on the topology. 

If we deal with spaces as the total computable functions or the computable real numbers, then a metric is at hand which allows putting stronger conditions on the convergence of sequences, e.g., we can prescribe its velocity. These conditions are important in order to be able to render the limit operator computable, however they are not compatible with wait-and-see strategies. Other algorithmic techniques like decision procedures are at hand instead. Also for spaces of this kind it is shown that effective operators are effectively sequentially continuous. However, we have not been able to present a uniform approach to the question whether effective operators are effectively sequentially continuous as we did in the case of effective continuity. It is even not clear whether this holds in general. A modification of Friedberg's example shows that effective operators are not sequentially continuous in general. But this still leaves open the possibility that they are effectively sequentially continuous as we are dealing with computable sequences only in this case.

As is well known, limits of point sequences in a $T_0$ space are not uniquely determined. In the joint paper \cite{sy84} we had to make a special assumption to handle this problem. Later, in \cite{sp98} we based our approach on filter convergence to get rid of it. In both cases we had to assume that one can effectively pass from a computable enumeration of the sequence elements and/or a filter base to the points they converge to. The relationship between both conditions will be studied as well.

The paper is organized as follows:  Section~\ref{basic} contains basic definitions. In Section~\ref{eff} notions and results from the theory of effective spaces are recalled. A new construction of an acceptable numbering is given. Important special cases of such spaces are considered in Section~\ref{speccase}.
The condition of a numbering having a limit algorithm and the existence of such numberings is discussed in Section~\ref{lim}. In Section~\ref{cont} the relationship between effective continuity notions of different strength is investigated, in particular the connection between effective continuity and effective sequential continuity. Finally, in Section~\ref{map}, the question of when an effective map is effectively sequentially continuous is examined.

\section{Basic definitions}\label{basic}

In what follows, let $\langle\;,\;\rangle:\omega^{2}\rightarrow\omega$ 
be a computable pairing function with corresponding  projections $\pi_1$
and $\pi_2$ such that $\pi_{i}(\langle a_{1},a_{2}\rangle )=a_{i}$. We
extend the pairing function to an $n$-tupel encoding in the usual way.
Let $P^{(n)}$ ($R^{(n)}$) denote the set of all $n$-ary partial
(total) computable functions, and let $W_{i}$ be the domain of the
$i$th partial computable function $\varphi_{i}$ with respect
to some G\"{o}del numbering $\varphi$. We let
$\varphi_{i}(a)\conv$ mean that the computation of
$\varphi_{i}(a)$  stops, $\varphi_{i}(a)\conv \in C$
that it stops with value in $C$, and
$\varphi_{i}(a)\conv_{n}$ that it stops within $n$
steps. In the opposite cases we write
$\varphi_{i}(a)\nconv$  and
$\varphi_{i}(a)\nconv_{n}$ respectively.  
Moreover, we write $\pfun{F}{X}{Y}$ to mean that $F$ is a partial function from set $X$ into set $Y$ with domain $\dom(F)$.

A {\em (partial) numbering\/} $\nu$ of a set $S$
is a partial map $\pfun{\nu}{\omega}{S}$ (onto). The value of $\nu$ at $n \in \dom(\nu)$ is denoted by $\nu_n$. If $s \in S$ and $n \in \dom(\nu)$ with $\nu_n = s$, then $n$ is said to be an \emph{index} of $s$.  Numberings $\nu$ with $\dom(\nu) = \omega$,  are called \emph{total}. Note that instead of numbering we also say \emph{indexing}. 
 
\begin{definition}\label{df-meq}
Let $\nu, \kappa$ be numberings of set $S$.
\begin{enumerate}

\item $\nu \le \kappa$, read $\nu$ is {\em reducible\/} to $\kappa$,
if there is some function $g \in P^{(1)}$ with $\dom(\nu)
\subseteq \dom(g)$, $g(\dom(\nu)) \subseteq \dom(\kappa)$, and $\nu_m =
\kappa_{g(m)}$, for all $m \in \dom(\nu)$.

\item $\nu \equiv \kappa$, read $\nu$ is {\em equivalent\/} to
$\kappa$, if $\nu \le \kappa$ and $\kappa \le \nu$.
\end{enumerate}
\end{definition}

A subset $X$ of $S$ is {\em completely enumerable}, if there is a computably enumerable set $A \subseteq \omega$ such that $\nu_i \in X$ if and only if $i \in A$, for all $i \in \dom(\nu)$.
$X$ is {\em enumerable}, if there is a computably enumerable set $E \subseteq \dom(\nu)$ such that $X = \set{\nu_i}{i \in E}$. 

Thus, $X$ is enumerable if we can enumerate a subset of the index set of $X$ which contains at least one index for every element of $X$, whereas $X$ is completely enumerable if we can enumerate all indices of elements of $X$ and perhaps some numbers which are not used as indices by  numbering $\nu$. 

\begin{definition}\label{def-fct-eff}
A map $\fun{F}{S}{T}$ between sets $S$ and $T$ with numberings $\nu$ and $\kappa$, respectively, is \emph{effective}, if there is a function $f \in P^{(1)}$ such that $f( i )\conv \in \dom(\kappa)$ and $F(\nu_i) = \kappa_{f(i)}$, for all $i \in \dom(\nu)$. Function $f$ is said to \emph{track} $F$ and any G\"odel number of $f$ is called \emph{index} of $F$.
\end{definition}

Note that the preimage of a completely enumerable set with respect to an effective map is completely enumerable again. 

A sequence $(y_a)_{a \in \omega}$ of elements of $S$ is \emph{computable} if there is some function $g \in R^{(1)}$ with $\range(g) \subseteq \dom\nu)$ so that $y_a = \nu_{g(a)}$, for all $a \in \omega$. Every G\"odel number of $g$ is called \emph{index} of $(y_a)_a$.
Let $\omega$ be enumerated by its identity. Then the  computable sequences in $S$ are the effective maps from $\omega$ to $S$. 

Note that the effectivity notions introduced so far depend on numbering $\nu$ (and/or numberings $\nu$ and $\kappa$ in the case of Definition~\ref{def-fct-eff}). In what follows we will fix certain numberings and consider them as being part of the effective structure we are considering. Therefore, we refrain from always denoting this dependency, in particular from using notation that would make it explicit.

\section{Effective spaces}\label{eff}

Let $\mathcal{T} = (T,\tau)$ be a countable topological $T_0$ space with a countable basis $\mathcal{B}$. As has been demonstrated by the author in a series of papers (Spreen 1995, 1996, 1998, 2001a, 2001b, 2010, 2014), topological spaces of this kind are well suited for effectivity considerations.

Assume further that $B$ is a total numbering of $\mathcal{B}$. In the applications we have in mind the basic open sets can be described in a finite way. $B$ is then obtained by encoding the finite descriptions. If we want to deal with the points and open sets of space $\mathcal{T}$ in an effective way,  the interplay between both should at least be such that we can effectively list the points of each basic open set, uniformly in its index.

\begin{definition}
Let $\mathcal{T} = (T, \tau)$ be a countable topological $T_0$ space with countable basis $\mathcal{B}$, and let $x$ and $B$ be numberings of $T$ and $\mathcal{B}$, respectively, such that $B$ is total. We say that $x$ is \emph{computable} if there is some computably enumerable set $L \subseteq \omega$ such that for all $i \in \dom(x)$ and all $n \in \omega$,
\[
\pair{i,n} \in L \Longleftrightarrow x_i \in B_n.
\]
\end{definition}
Thus, $x$ is computable if and only if all basic open sets $B_n$, are completely enumerable, uniformly in $n$.

We consider the numberings $B$ and $x$ as being part of the topological structure. 

As said, in the applications we have in mind basic open sets can be described in a finite way and the indexing $B$ is then obtained by an encoding of the finite descriptions. Moreover, in these cases there is a canonical relation between the (code numbers of the) finite descriptions which is stronger than the usual set inclusion between the described sets. This relation is computable enumerable, which is not true for set inclusion, in general. 
\begin{definition}\label{df-si}
Let $\prec_B$ be a transitive binary relation on $\omega$. We say that:
\begin{enumerate} 
\item\label{df-si-1} $\prec_B$ is a \emph{strong inclusion}, if for all $m, n \in \omega$, from $m \prec_B n$ it follows that $B_m \subseteq B_n$.

\item\label{df-si-2} $\mathcal{B}$ is a \emph{strong basis}, if $\prec_B$ is a strong inclusion and for all $z \in T$ and $m, n \in \omega$ with $z \in B_m \cap B_n$ there is some $a \in \omega$ such that $z \in B_a$, $a \prec_B m$ and $a \prec_B n$. 

\end{enumerate}
\end{definition} 

In what follows, we always assume that $\prec_B$ is a strong inclusion with respect to which $\mathcal{B}$ is a strong basis.

\begin{definition}
 Space $\cal T$ is {\em effective}, if the property of being a strong basis holds effectively, which means  that there exists a function $sb\in P^{(3)}$ such that for $i \in \dom(x)$  and $m$, $n\in \omega$  with $x_{i}\in B_{m}\cap B_{n}$, $sb(i,m,n)\mathclose\downarrow$, $x_{i}\in B_{sb(i,m,n)}$,  $sb(i,m,n)\prec_B m$, and $sb(i,m,n)\prec_B n$.
\end{definition}

\begin{lemma}[Spreen 1998]\label{lem-eff}
Let $x$ be computable and $\prec_B$ be computably enumerable. Then $\mathcal{T}$ is effective.
\end{lemma}

As is well known, each point $y$  of a $T_{0}$ space is uniquely determined by its neighbourhood filter $\mathcal{N}(y)$ and/or a base of
it.
\begin{definition}
Let $\cal H$ be a filter. A nonempty subset $\cal F$ of $\cal H$ is
called {\em strong base\/} of $\cal H$ if the following two
conditions hold:
\begin{enumerate}
\item \label{filt_filtered} For all $m$, $n \in\omega$ with $B_m$,
$B_n \in {\cal F}$ there is some index $a \in\omega$ such that $B_a
\in {\cal F}$, $a\prec_B m$, and $a\prec_B n$.

\item \label{filt_sub} For all $m \in\omega$ with $B_m \in {\cal H}$
there is some index $a\in\omega$ such that $B_a \in{\cal F}$ and $a
\prec_B m$.
\end{enumerate}
\end{definition}

If $x$ is computable, a strong base of basic open sets can effectively
be enumerated for each neighbourhood filter. Here, we are interested in enumerations that proceed in a normed way. 

\begin{definition}
An enumeration $( B_{f(a)} )_{a\in\omega}$  with $f:\omega\rightarrow\omega$  is said to be {\em normed\/} if $f$ is decreasing with respect to $\prec_B$. If $f$ is computable, it is also called {\em computable\/} and any G\"{o}del number of $f$ is said to be an {\em index\/} of it.
\end{definition}

In case $( B_{f(a)})_a$ is normed and enumerates a strong base of the neighbourhood filter of some point, we say it {\em converges\/} to that point. 

Recall here that because of the $T_0$ requirement every point is uniquely determined by a base of its neighbourhood filter. So, if a normed enumeration of basic open sets converges to a point, the point is uniquely determined by the enumeration. This is unlike the case of point sequences where limits need not be uniquely determined in general.

\begin{lemma}[Spreen 1998]\label{lm-normseq}
Let $\mathcal{T}$ be effective and $x$ be computable. Then there is a function $q\in R^{(1)}$ such that for each $i\in \dom(x)$, $q(i)$ 
is an  index of a normed computable enumeration of basic open sets converging to  $x_{i}$. 
\end{lemma}

We not only want  be able to generate normed recursive enumerations of basic open sets converging to a given point, but conversely, we need also be able to pass effectively from such enumerations to the point they converge to.

\begin{definition}
Let $x$ be a numbering of $T$. We say that:
\begin{enumerate}
\item $x$ {\em allows effective limit passing\/} if there is a function $pt\in P^{(1)}$ such that, if $m$ is an index of a  normed computable
enumeration of basic open sets  converging to a point $y\in T$, then $pt(m)\mathclose\downarrow \in \dom(x)$ and $x_{pt(m)} =
y$.

\item $x$ is {\em acceptable\/} if it allows effective limit passing and is computable.
\end{enumerate}
\end{definition}

\begin{lemma}[Spreen 1998]\label{lem-numb-red}
Let $x', x''$ be numberings of $T$.  Then the following three statements hold:
\begin{enumerate}
\item If $x'$ is computable and $x'' \le x'$, then $x''$ is computable as well.

\item If $x'$ allows effective limit passing and $x' \le x''$, then also $x''$ allows effective limit passing.

\item If $x'$ is computable, $\mathcal{T}$ effective with respect to $x'$, and $x''$ allows effective limit passing, then $x' \le x''$.
\end{enumerate}
\end{lemma}

The next result is now a consequence of Lemma~\ref{lem-eff}.

\begin{corollary}\label{cor-acc-eq}
Let $\prec_B$ be computably enumerable and $x$ acceptable. Then, for any other numbering $x'$ of $T$, $x'$ is acceptable exactly if $x$ and $x'$ are equivalent.
\end{corollary}

We will now give an example of an acceptable numbering that  shall be used again later.

\begin{proposition}\label{prop-num}
Let $\mathcal{T}$ be such that $\prec_B$ is computably enumerable and the neighbourhood filter of each point in $T$ has an enumerable strong base of basic open sets. Then, $T$ has an acceptable numbering.
\end{proposition}
\begin{proof}
If $\set{B_n}{n \in W_e}$ is a strong base of the neighbourhood filter of some point $y \in T$, set $\bar{x}_e = y$. Otherwise, let $\bar{x}$ be undefined. Because of the assumption, $\bar{x}$ is a numbering of $T$. Let $L = \set{\pair{e,n}}{n \in W_e}$. Then $L$ is computably enumerable. Moreover, we have for $i \in \dom(\bar{x})$ that
\[
\bar{x}_i \in B_n \Longleftrightarrow (\exists m \in W_i) m \prec_B n  \Longleftrightarrow (\exists m) \pair{i,m} \in L \wedge m \prec_B n,
\]
which shows that $\bar{x}$ is computable.

Next, let $m$ be an index of a normed computable enumeration of basic open sets converging to some point $y \in T$. Then $\set{B_n}{n \in \range(\varphi_m)}$ is a strong base of $\mathcal{N}(y)$. Hence, $y = \bar{x}_{t(m)}$, where $t \in R^{(1)}$ is such that $W_{t(a)} = \range(\varphi_a)$, showing that $\bar{x}$ also allows effective limit passing.
\end{proof}

For basic open sets $B_n$,  let 
\[
\hl(B_n) = \bigcap\set{B_a}{n \prec_B a}.
\]
Sometimes, when we need to choose certain elements in $B_n$, we may not be able to find them in $B_n$, but then we want to find them as close to $B_n$ as possible.

Let $X$ be an subset of $T$. A typical situation in many proofs is that we need to show for some basic open set $B_e$ that $B_e \subseteq X$. We would  try  a proof by contradiction and assume that $B_e \not\subseteq X$. Then we would \emph{choose}, uniformly in $e$ and perhaps some index of $X$, an element $z \in B_e \setminus X$ and derive a contradiction. In a non-effective setting the Axiom of Choice permits proceeding in this way. In an effective context, however, we have to effectively find such a \emph{witness} $z$. The situation particularly occurs in continuity proofs. In this case $X$ is the preimage of a basic open set $B'_n$ under a map $\fun{F}{T}{T'}$, where $\mathcal{T'} = (T', \tau')$ is a further countable $T_0$ space with countable basis $\mathcal{B'}$, a total numbering $B'$ of $\mathcal{B'}$, and an indexing $x'$ of $T'$. 

\begin{definition}[Spreen and Young 1984]\label{df-witt}
$F$ has \emph{a witness for noninclusion}, if there is a pair $(s, r) \in P^{(2)} \times  P^{(3)}$, the \emph{noninclusion witness},  such that for $i \in \dom(x)$ and $e, n \in \omega$ the following hold:
\begin{enumerate}
\item\label{df-wit-1}  If $F(x_i) \in B'_n$, then $s( i , n)\conv$ so that $F(x_i) \in B'_{s(i,n)} \subseteq B'_n$.

\item\label{df-wit-2} If, in addition, $F(B_e) \not\subseteq B'_n$, then also $r(i,e,n)\conv \in\dom(x)$ with $x_{r(i,e,n)} \in \hl(B_e) \setminus F^{-1}(B'_{s(i,n)})$.

\end{enumerate}
\end{definition}

To understand this definition, suppose that $F(x_i) \in B'_n$, but the neighbourhood $B_e$ of $x_i$ does not map into the neighbourhood $B'_n$ of $F(x_i)$. Then we can effectively find a (possibly) smaller neighbourhood $B'_a$ with $F(x_i) \in B'_a \subseteq B'_n$ and a point $x_{r(i,e,n)}$ which, under $F$, maps, not necessarily outside $B'_n$, but at least outside $B'_a$. Obviously, $a$ may depend on $i$ and $n$.

\section{Special cases}\label{speccase}

In this section we introduce some important standard examples of effective $T_0$ spaces: constructive domains and constructive metric spaces. Domains play a major role in theoretical computer science, particularly in programming language semantics~\cite{gu92,ac98,gi03} and exact real number computation~\cite{ed97}. Metric spaces, on the other hand, are well known from applied mathematics. Topologically, as well as computationally, both spaces are quite different: In general domains satisfy only $T_0$ separation, whereas metric spaces are Hausdorff.   

As is well known, $T_0$ spaces come equipped with a canonical order $\le_\tau$, called \emph{specialization order}: For $y, z \in T$,
\[
y \le_\tau z \Longleftrightarrow (\forall n \in \omega) [y \in B_n \Rightarrow z \in B_n].
\]
Every open set is upwards closed under the specialization order and continuous maps are monotone with respect to it. 

As has already been pointed out, limits of point sequences in $T_0$ spaces  need not be uniquely determined. In case the sequence is monotonically increasing with respect to the specialization order, every sequence element is a limit.

\subsection{Constructive domains}\label{domains}

Let $Q = (Q, \sqsubseteq)$ be a partial order with least element. A nonempty subset $S$ of $Q$ is \emph{directed}, if for all $y_1, y_2 \in S$ there is some $u \in S$ with $y_1, y_2 \sqsubseteq u$. The {\em way-below relation\/} $\ll$ on $Q$ is defined as follows: $y_1 \ll y_2$ if for every directed subset $S$ of $Q$ the least upper bound of which exists in Q, the relation $y_2 \sqsubseteq \bigsqcup S$  implies  the existence of an element  $u \in S$   with $y_1 \sqsubseteq u$. Note that $\ll$ is transitive.  

A subset $Z$ of $Q$ is a {\em basis\/} of $Q$, if for any  $y \in Q$ the set $Z_y = \{\,z \in Z \mid z \ll y \,\}$ is directed and $y = \bigsqcup Z_y$. A partial order that has a basis is called {\em continuous}. 

Now, assume that $Q$ is countable and let $x$ be an indexing of $Q$. Then $Q$ is {\em constructively d-complete\/}, if each of its enumerable directed subsets has a least upper bound in $Q$.  Let $Q$ be constructively d-complete and continuous with basis $Z$. Moreover, let $\beta$ be a total numbering of $Z$. Then $(Q, \sqsubseteq, Z, \beta, x)$ is said to be a {\em constructive domain}, if the restriction of the way-below relation to $Z$ as well as all sets $Z_y$, for $y \in Q$, are completely enumerable with
respect to the indexing $\beta$, and $\beta \le x$. 

A numbering $x$ of $Q$ is said to be {\em admissible}, if the set $\{\,\langle i,j \rangle \mid \beta_i \ll x_j\,\}$ is computably enumerable and there is a function  $d \in R^{(1)}$ such that for all indices $i \in \omega$ for which $\beta(W_i)$ is directed, $x_{d(i)}$ is the least upper bound of $\beta(W_i)$. As shown in \cite{wd80}, such numberings always exist. They can even be chosen as total. 

Partial orders come with several natural topologies. In the applications we have in mind, one is mainly interested in the {\em Scott topology} $\sigma$: a subset $X$ of $Q$ is open in $\sigma$, if it is upwards closed with respect to the partial order and intersects each enumerable directed subset of $Q$ of which it contains the least upper bound. 

The Scott topology satisfies $T_0$ separation, but in general not $T_1$. The partial order on $Q$ coincides with the specialization order defined by the topology in this case \cite{gi03}. Moreover, least upper bounds of monotonically increasing sequences are limits; in particular they are maximal limits.  

In the case of a constructive domain the Scott topology is generated by the sets $B_n = \set{y \in Q}{\beta_n \ll y}$ with $n \in \omega$. It follows that ${\cal Q} =(Q, \sigma)$ is a countable $T_0$-space with countable basis.  Obviously, every admissible numbering is computable.

Define 
\[
m \prec_B n \Leftrightarrow \beta_n \ll \beta_m.
\] 
Then $\prec_B$ is a strong inclusion with respect to which the collection of all $B_n$ is a strong basis. Because the restriction of $\ll$ to $Z$ is completely enumerable, $\prec_B$ is computably enumerable.\ It follows that ${\cal Q}$ is effective. Moreover, each admissible indexing allows effective limit passing, i.e., it is acceptable. Conversely, every acceptable numbering of $Q$ is admissible. 

Note that the set $P^{(1)}$ of partial computable functions, ordered by $f \sqsubseteq g$, if $g$ extends $f$, is a constructive domain. The finite functions form a basis and each G\"odel numbering is admissible. 

As a further example consider the set $\mathbb{I}[0,1]_c$ of all closed subintervals of $[0,1]$ with computable real numbers as endpoints. Ordered by converse set inclusion $\mathbb{I}[0,1]_c$ is a constructive domain with the closed intervals having rational endpoints as basis. The computable real numbers $z$ in $[0,1]$ correspond to the one-point intervals $[z,z]$.

Domains are usually introduced as an ordered structure. The basic notions are order-theoretic, topology is introduced only at a later step.
In order to provide a (more) topological approach to domain theory, Er\v{s}ov (1972, 1973, 1975, 1977) introduced $A$- and $f$-spaces. They are not required to be complete. Constructive $A$- and $f$-spaces as introduced in \cite{sp98} are further examples of effective $T_0$ spaces.

An essential property of constructive domains, just as of Er\v{s}ov's $A$- and $f$-spaces, is that their canonical topology has a basis with every basic open set $B_n$ being an upper set generated by a point which is not necessarily included in $B_n$, but in $\hl(B_n)$.

\begin{definition}\label{def-point}
A countable $T_0$ space $\mathcal{T}$ with countable basis $\mathcal{B}$ and numberings $x$ and $B$ of $T$ and $\mathcal{B}$, respectively, is \emph{effectively pointed}, if there is a function $\pd \in P^{(1)}$ such that for all $n \in \omega$ with $B_n \not= \emptyset$, $\pd(n)\conv \in \dom(x)$, $x_{\pd(n)} \in \hl(B_n)$ and $x_{\pd(n)} \le_\tau z$, for all $z \in B_n$.
\end{definition}

Note that $\set{x_a}{a \in \range(\pd)}$ is dense in $\mathcal{T}$. For a constructive domain $(Q, \sqsubseteq, Z, \beta, x)$, let $\pd \in R^{(1)}$ with $\beta = x \circ \pd$. It follows that $\mathcal{Q}$ is effectively pointed.

\subsection{Constructive metric spaces}\label{consmet}

Whereas domains as well as $A$- and $f$-spaces typically do not satisfy $T_2$ separation, in this section we will consider the standard example of an effective Hausdorff space.

Let ${\cal M} = (M, \delta)$ be a countable separable metric space and  $\beta$ be a total numbering of its dense subset $M_0$. As is well-known, the collection of sets $B_{\langle i, m \rangle} = \{\,y \in M \mid \delta(\beta_i, y) < 2^{-m}\,\}$ ($i$, $m \in \omega$) is a basis of the canonical Hausdorff topology $\Delta$ on $M$.

Define 
\[\langle i,m\rangle  \prec_B \langle j,n\rangle  
\Leftrightarrow  \delta(\beta_i,\beta_j) + 2^{-m} < 2^{-n}.
\]
Using the triangle inequality it is readily verified that $\prec_B$ is  a strong inclusion and the collection of all $B_a$ is a strong basis. 

\begin{definition}\label{met}
$\cal M$ is said to be {\em constructive}, if the sets 
\[
\set{\pair{i,j, a,n}}{\delta(\beta_i, \beta_j) < a \cdot 2^{-n}} \quad\text{and}\quad
\set{\pair{i,j,a,n}}{\delta(\beta_i, \beta_j) > a \cdot 2^{-n}}
\]
are computably enumerable, and the neighbourhood filter of each point has an enumerable strong base of basic open sets.
\end{definition}

Obviously, $\prec_B$ is computably enumerable in this case.

Well-known examples of constructive metric spaces include $\mathbb{R}^n_c$, that is the space of all $n$-tuples of computable real numbers with the Euclidean or the maximum norm; Baire space, that is the set $R^{(1)}$ of all total computable functions with the Baire metric \cite{ro67}; and the set $\omega$ with the discrete metric. By using an effective version of Weierstra{\ss}'s approximation theorem \cite{pr89} and Sturm's theorem \cite{st35} it can be shown that $C_c[0,1]$, the space of all computable functions from $[0,1]$ to $\mathbb{R}$ with the supremum norm \cite{pr89}, is a constructive metric space too. A proof of this result and further examples can be found in  \cite{bl97}.

\section{Limit algorithms}\label{lim}

In this note we assume each space to come with a rich collection \textsc{Seq} of canonical computable sequences with the following properties:
\begin{enumerate}
\item\label{prop-seq-1} All computable sequences that are monotonically increasing with respect to the specialization order are in \textsc{Seq}.

\item\label{prop-seq-2} There is a function $p \in R^{(1)}$ such that for each index $m$ of a normed computable enumeration of basic open sets converging, say, to $y \in T$, $p(m)$ is an index of a computable sequence of points in \textsc{Seq} satisfying the subsequent two conditions:
\begin{enumerate}

\item\label{prop-seq-2-1} $x_{\varphi_{p(m)}(a)} \in \hl(B_{\varphi_m(a)})$.

\item\label{prop-seq-2-2} In case $\mathcal{T}$ does not satisfy $T_1$ separation, there exist at most finitely many $a \in \omega$ with $x_{\varphi_{p(m)}(a)} \in B_n$, for every basic open set $B_n \not\in \mathcal{N}(y)$.

\end{enumerate} 

\item\label{prop-seq-3} If $(y_a)_a$ is in \textsc{Seq}, then, for every $\bar{a} \in \omega$, $(y'_a)_a$ is in \textsc{Seq} as well, where $y'_a = y_a$, for $a < \bar{a}$, and $y'_a = y_{\bar{a}}$, otherwise.

\end{enumerate}

In the case of effectively pointed spaces, \textsc{Seq} consists of all computable monotonically increasing sequences. Let $\pd \in P^{(1)}$ be as in Definition~\ref{def-point}. Then, if $m$ is an index of a converging normed computable enumeration of basic open sets, we have for all $a \in \omega$ that $x_{\pd(\varphi_m(a))} \in \hl(B_{\varphi_m(a)})$. In order to see that also the second condition holds, note that by definition each open set is upwards closed under the specialization order. Therefore, if for some $a \in \omega$, $x_{\pd(\varphi_m(a))} \in B_n$, then for all $a' > a$, also $x_{\pd(\varphi_m(a'))} \in B_n$.
Thus, $p \in R^{(1)}$ with $\varphi_{p(m)}(a) = \pd(\varphi_m(a))$ has the desired property.

In the metric case we let \textsc{Seq} be the set of all computable regular Cauchy sequences, where a Cauchy sequence $(y_a)_a$ is \emph{regular} (or, \emph{fast}), if $\delta(y_m, y_n) < 2^{-m}$, for all $n \ge m$. Instead, one could also take the set of all computable Cauchy sequences with a computable Cauchy criterion (cf.\ \cite{mo65}). If $x$ is such that for some $g \in R^{(1)}$, $\beta = x \circ g$, and $m$ is an index of a normed computable enumeration of basic open sets, choose $p \in R^{(1)}$ with $\varphi_{p(m)}(a) = g(\pi_1(\varphi_m(a)))$, for $a, m \in \omega$.

In Section \ref{eff} as well as in other papers we based our approach to the computation of limits on filter convergence. In the earlier paper \cite{sy84}, however, we used point sequence convergence. One of the main reasons for moving to filters was that in $T_0$ spaces the limit of a point sequence is not uniquely determined, in general: if $y$ is a limit point, every $z$ with $z \le_\tau y$ is a limit point as well. We denote the set of limit points of a sequence $(y_a)_a$ by $\Lim_a y_a$.

Because every normed enumeration $(B_{\varphi_m(a)})_a$ of basic open sets converging to a point $y \in T$ is a base of $\mathcal{N}(y)$, it follows with Property~(\ref{prop-seq-2-1}) that $y \in \Lim_a x_{\varphi_{p(m)}(a)}$. With (\ref{prop-seq-2-2}) we moreover obtain that $y = \max_{\le_\tau} \Lim_a x_{\varphi_{p(m)}(a)}$.

To see this, assume there is some $z \in \Lim_a x_{\varphi_{p(m)}(a)}$ with $z \not\le_\tau y$. Then there exists $B_n \in \mathcal{N}(z) \setminus \mathcal{N}(y)$. It follows for some $\bar{a} \in \omega$ that $x_{\varphi_{p(m)}(a)} \in B_n$, for all $a \ge \bar{a}$, which is impossible by Property~(\ref{prop-seq-2-2}).

\begin{definition}\label{df-la}
A numbering $x$ of $T$ has a \emph{limit algorithm}, if there is a function $\li \in P^{(1)}$ such that the following four conditions hold, for all indices $m, m'$ of convergent\footnote{The convergence of point sequences in \textsc{Seq} is to be understood in the usual topological way.} sequences in \textsc{Seq}:
\begin{enumerate}

\item\label{df-la-0} $\li(m)\conv \in \dom(x)$.

\item\label{df-la-1} $x_{\li(m)} \in \Lim_a x_{\varphi_m(a)}$

\item\label{df-la-2} If, for some $\bar{a} \in \omega$, $x_{\varphi_m(a)} = x_{\varphi_m(\bar{a})}$, for all $a \ge \bar{a}$, then $x_{\li(m)} = x_{\varphi_m(\bar{a})}$.

\item\label{df-la-3} If $\Lim_a x_{\varphi_m(a)} = \Lim_a x_{\varphi_{m'}(a)}$, then $x_{\li(m)} = x_{\li(m')}$.
\end{enumerate}
\end{definition}

The property of being able to compute limits may seem unusually strong, in particular in the light of Specker's result that being able to compute limits of arbitrary computable (not necessarily fast) converging sequences is equivalent to deciding the halting problem~\cite{sk49}. But note that here this property is only required to hold for canonical sequences in \textsc{Seq}. Note further that we are not interested in computing just one limit of a given computable canonical sequence. Instead we want to compute a distinguished limit which in the case of monotonically increasing sequences is the maximal limit, if it exists. This will become clear in the sequel.

A typical technique in enumeration is to wait and see, i.e., to repeat what has already been enumerated till new information becomes available. This motivates the following condition.

\begin{definition}
A sequence $(y_a)_a$ in \textsc{Seq} is said to \emph{allow delaying}, if for all $\bar{a}, m \in \omega$ the sequence $(y'_a)_a$ with $y'_a = y_a$, for $a < \bar{a}$, $y'_a = y_{\bar{a}}$, for $\bar{a} \le a \le \bar{a} + m$, and $y'_a = y_{a - m + 1 }$,  otherwise, is in \textsc{Seq} as well.  
\end{definition} 

We start with a general result.

\begin{proposition}\label{prop-liprop1}
Let $x$ have a limit algorithm. Moreover, let $X$ be a completely enumerable subset of $T$ and $(y_a)_a$ a convergent sequence in \textsc{Seq} that allows delaying. Then for every index $m$ of $(y_a)_a$ and each number $\bar{a} \in \omega$, if $y_{\bar{a}} \in X$, also $x_{\li(m)} \in X$.
\end{proposition}
\begin{proof}
Let $W_e$ witness the complete enumerability of $X$ and set $g(\bar{b}) = \mu c > \bar{a}: \li(\bar{b})\conv_c \wedge \varphi_e(\li(\bar{b}))\conv_c$. By the recursion theorem there is then some $b \in \omega$ with
\[
\varphi_{b}(a) = \begin{cases}
                                           \varphi_m(a) & \text{ if $a \le \bar{a}$,}\\
				\varphi_m(\bar{a}) & \text{ if $ a > \bar{a}$, and $\li(b)\nconv_a$ or  $\varphi_e(\li(b))\nconv_a$,}\\
				\varphi_m(\bar{a} + a - g(b) + 1)  & \text{ if $a > \bar{a}$, $\li(b)\conv_a$, and $\varphi_e(\li(b))\conv_a$.}
			       \end{cases}
\]	

Suppose that $g(b)\nconv$. Because of Property~\ref{prop-seq-3} of \textsc{Seq}, the sequence $(x_{\varphi_{b}(a)})_a$ is in \textsc{Seq} in  this case. Moreover, it converges to $y_{\bar{a}}$. With Condition~\ref{df-la}(\ref{df-la-2}) we therefore obtain that $\li(b)\conv$ and $x_{\li(b)} = y_{\bar{a}}$. By our assumption, $y_{\bar{a}} \in X$, i.e., $\varphi_m(\bar{a}) \in W_e$. It follows that $\li(b) \in W_e$ as well, which means that $g(b)\conv$, a contradiction.

So we have that both, $\li(b)\conv$ and $\varphi_e(\li(b))\conv$. Since $(y_a)_a$ allows delaying, it follows that the just defined sequence with index  $b$ is in \textsc{Seq}, also in this case. Moreover, $\Lim_a x_{\varphi_{b}(a)} = \Lim_a y_a$ and hence, by Condition~\ref{df-la}(\ref{df-la-3}), $x_{\li(m)} = x_{\li(b)}$. As a further consequence, $\li(b) \in W_e$, which means that $x_{\li(b)} \in X$. This shows that $x_{\li(m)} \in X$.
\end{proof}

Let $u, z \in T$ with $u \le_\tau z$. Then the sequence with $y_a = u$, for $a \le \bar{a}$ and $y_a = z$, otherwise, for some $\bar{a} \in \omega$, is in \textsc{Seq}, by Condition~\ref{prop-seq-1} for \textsc{Seq}, and obviously allows delaying. Thus, if $u \in X$, then $z \in X$ as well.

\begin{corollary}\label{cor-cecl}
Let $x$ have a limit algorithm.  Then each completely enumerable subset of $T$ is upwards closed under the specialization order.
\end{corollary}

Next, suppose that $\mathcal{T'} = (T', \tau')$ is a further countable $T_0$ space with countable basis $\mathcal{B'}$ and numberings $x'$ and $B'$ of $T'$ and $\mathcal{B'}$, respectively, such that $B'$ is total. Moreover,  recall that the preimage of a completely enumerable set under an effective map is completely enumerable again.

\begin{corollary}\label{cor-effmon}
Let $x$ have a limit algorithm and $x'$ be computable. Then every effective map $\fun{F}{T}{T'}$ is monotone with respect to the specialization order. 
\end{corollary}

In case that $\mathcal{T}$ is effectively pointed, this result implies that every effective map $\fun{F}{T}{T'}$ has a witness for noninclusion. In what follows we need that noninclusion witnesses respect the canonicity of the sequences in \textsc{Seq}. Let to this end $p \in R^{(1)}$ be as in Condition~\ref{prop-seq-2} for \textsc{Seq}.

\begin{definition}\label{df-compat}
A noninclusion witness $(s,r) \in P^{(2)} \times P^{(3)}$ for a map $\fun{F}{T}{T'}$ is called \emph{appropriate} if for each index $m$ of a computable normed enumeration of basic open sets converging to $x_i$, every $n \in \omega$ so that $F(x_i) \in B'_n$, and each $e \in \omega$ such that $F(B_{\varphi_m(e)}) \not\subseteq B'_n$, the sequence $(y_a)_a$ with $y_a = x_{\varphi_{p(m)}(a)}$, for $a < e$, and $y_a = x_{r(i,\varphi_m(e),n)}$, otherwise, is in \textsc{Seq}. 
\end{definition}

A requirement of this kind was already used when the condition of having a witness for noninclusion was introduced in \cite{sy84}, but dropped later.

\begin{corollary}\label{cor-pt-wit}
Let $\mathcal{T}$ be effectively pointed, $x$ have a limit algorithm and $x'$ be computable. Then every effective map $\fun{F}{T}{T'}$ has an appropriate witness for noninclusion.
\end{corollary}
\begin{proof}
Set $s(i,m) = m$ and $r(i,e,n) = \pd(e)$. Then $x_{r(i,e,n)} \in \hl(B_e)$. Suppose that $F(x_{r(i,e,n)}) \in B'_n$, i.e., $F(x_{\pd(e)})  \in B'_n$. By the preceding corollary it follows that $F(B_e) \subseteq B'_n$. Because of Properties~(\ref{prop-seq-2}) and (\ref{prop-seq-3}) of \textsc{Seq} the witness is also appropriate. 
\end{proof}

\begin{corollary}\label{cor-liprop}
 Let $x$ be computable and have a limit algorithm. Then, for every index $m$ of a convergent sequence in \textsc{Seq} that allows delaying, the following two statements hold:
\begin{enumerate}

\item\label{cor-liprop-1} $x_{\varphi_m(a)} \le_\tau x_{\li(m)}$, for all $a \in \omega$.

\item\label{cor-liprop-2} $y \le_\tau x_{\li(m)}$, for all $y \in \Lim_a x_{\varphi_m(a)}$.

\end{enumerate}
\end{corollary}
\begin{proof} 
(\ref{cor-liprop-1}) Let $a, n \in \omega$ with $x_{\varphi_m(a)} \in B_n$. Then it follows with the preceding proposition that $x_{\li(m)} \in B_n$ as well. Thus, $x_{\varphi_m(a)} \le_\tau x_{\li(m)}$.

(\ref{cor-liprop-2}) Let $y \in \Lim_a x_{\varphi_m(a)}$ and $n \in \omega$ with $y \in B_n$. Then there is some $\bar{a} \in \omega$ so that $x_{\varphi_m(a)} \in B_n$, for all $a \ge \bar{a}$. By Statement~\ref{cor-liprop-1} and the definition of the specialization order it follows that also $x_{\li(m)} \in B_n$, which shows that $y \le_\tau x_{\li(m)}$.
\end{proof}

In case that the function $\li \in P^{(1)}$ in Definition~\ref{df-la} also satisfies the condition in Statement~\ref{cor-liprop}(\ref{cor-liprop-2}), we say that the limit algorithm \emph{computes maximal limits}.

It is well known that if space $\mathcal{T}$  satifies  $T_1$ separation  its specialization order coincides with the identity relation on $T$. Under the assumptions of the above corollary we therefore obtain that convergent sequences in \textsc{Seq} that allow delaying must be constant. In other words,  except in the case of $T_0$ spaces that violate the $T_1$ condition, only trivial sequences in \textsc{Seq} satisfy the assumption.

As we will see next, the property in Corollary~\ref{cor-liprop}(\ref{cor-liprop-1}) is characteristic for limit algorithms.
 
\begin{lemma}\label{lm-propli}
Let $\li \in P^{(1)}$ such that for all indices $m$ of convergent sequences in \textsc{Seq}, $\li(m)\conv \in \dom(x)$ with $x_{\li(m)} \in \Lim_a x_{\varphi_m(a)}$ and $x_{\varphi_m(a)} \le_\tau x_{\li(m)}$, for all $a \in \omega$. Then $x$ has a limit algorithm.
\end{lemma}
\begin{proof}
As in the preceding proof we obtain that $y \le_\tau x_{\li(m)}$, for all $y \in \Lim_a x_{\varphi_m(a)}$. Hence, $x_{\li(m)} = \max_{\le_\tau} \Lim_a x_{\varphi_m(a)}$. If $m'$ is an index of a further converging sequence in \textsc{Seq} so that $\Lim_a x_{\varphi_{m'}(a)} = \Lim_a x_{\varphi_m(a)}$, we therefore have that 
\[
x_{\li(m')} = \max\nolimits_{\le_{\tau}} \Lim_a x_{\varphi_{m'}(a)} = \max\nolimits_{\le_\tau} \Lim_a x_{\varphi_m(a)} = x_{\li(m)}.
\]

If there is some $\bar{a} \in \omega$ such that $\varphi_m(a) = \varphi_m(\bar{a})$, for all $ a \ge \bar{a}$, then $x_{\varphi_m(\bar{a})} \in \Lim_a x_{\varphi_m(a)}$. Thus, $x_{\varphi_m(\bar{a})} \le_\tau x_{\li(m)}$. To see that also the converse inequality holds, let $n \in \omega$ with $x_{\li(m)} \in B_n$.  Since $x_{\li(m)}$ is a limit point of $(x_{\varphi_m(a)})_a$, there is some $\hat{a} \in \omega$ with $x_{\varphi_m(a)} \in B_n$, for all $a \ge \hat{a}$. In particular, we have that $x_{\varphi_m(\bar{a})} \in B_n$, which shows that $x_{\li(m)} \le_\tau x_{\varphi_m(\bar{a})}$.
\end{proof}

\begin{proposition}\label{pn-li}
Let $x$ be computable and all sequences in \textsc{Seq} allow delaying. Then $x$ has a limit algorithm if, and only if, there is some function $\li \in P^{(1)}$ so that for all indices $m$ of convergent sequences in \textsc{Seq},  $\li(m)\conv \in \dom(x)$ with $x_{\li(m)} \in \Lim_a x_{\varphi_m(a)}$ and $x_{\varphi_m(a)} \le_\tau x_{\li(m)}$, for all $a \in \omega$.
\end{proposition}

This gives us a hint of how to construct a numbering of $T$ that has a limit algorithm in the case of $T_0$ spaces that do not satisfy $T_1$ separation. We say that \textsc{Seq} \emph{has maximal limits} if $\Lim_a y_a$ has a greatest element, for each converging sequence $(y_a)_a$ in \textsc{Seq}.

As in the case of effectively pointed spaces, we now let \textsc{Seq} contain only sequences that monotonically increase with respect to the specialization order. Such sequences always allow delaying. As follows from the definition of the specialization order, every sequence element is a limit in this case and the greatest limit point, if it exists, is the least upper bound. 

In addition, we assume that $\prec_B$ is computably enumerable and the neighbourhood filter of each point has an enumerable strong base of basic open sets. As we have seen in Section~\ref{speccase}, the just mentioned spaces always satisfy these assumptions. Let $\bar{x}$ be the numbering constructed in Proposition~\ref{prop-num} and $\li \in R^{(1)}$ with
\[
W_{\li(m)} = \set{n \in \omega}{(\exists a \in \omega) n \in W_{\varphi_m(a)}}.
\]
Suppose that $m$ is an index of a monotonically increasing sequence with a largest limit element $y$. Then all $B_n$ with $n \in W_{\li(m)}$ contain $y$. On the other hand, if $B_n$ contains $y$ then there is some $a \in \omega$ with $x_{\varphi_m(a)} \in B_n$, as $y$ is a limit point. Thus, $n \in W_{\li(m)}$. It follows that $\set{B_n}{n \in W_{\li(m)}}$ is the set of all basic open sets containing $y$ and hence a strong base of $\mathcal{N}(y)$. So, $y = \bar{x}_{\li(m)}$, which shows that $\bar{x}$ has a limit algorithm.

\begin{proposition}\label{prop-t0la}
Let \textsc{Seq} contain only monotonically increasing sequences and have maximal limits. Moreover, let $x$ be acceptable, $\prec_B$ be computably enumerable, and the neighbourhood filter of each point in $T$ have an enumerable strong base of basic open sets. Then $x$ has a limit algorithm.
\end{proposition}
\begin{proof}
By assumption $x$ is computable and $\prec_B$ computably enumerable.  Moreover, as we have seen, the numbering $\bar{x}$ constructed in Proposition~\ref{prop-num} is acceptable and has a limit algorithm. Hence, $\mathcal{T}$ is effective with respect to $x$ as well as $\bar{x}$, because of Lemma~\ref{lem-eff}. With Corollary~\ref{cor-acc-eq} we therefore obtain that $x$ and $\bar{x}$ are equivalent. Obviously, the property of having a limit algorithm is inherited under equivalence.
\end {proof}

\begin{proposition}\label{prop-metla}
Let $\mathcal{M}$ be a constructive metric space. Then every acceptable numbering of $M$ has a limit algorithm.
\end{proposition}
\begin{proof}
Again it suffices to show that the numbering $\bar{x}$ constructed in Proposition~\ref{prop-num} has a limit algorithm. Let $(y_a)_a$ be a computable regular Cauchy sequence that converges to some point $y \in M$. Because of regularity we have that $\delta(y_a, y) \le 2^{-a}$, for all $a \in \omega$. Let $a >0$. Since $M_0$ is dense in $M$, there is some $\beta_i \in M_0$ such that $\delta(\beta_i, y_a) < 2^{-a}$. By the triangular inequation it then follows that $\set{u \in M}{\delta(y_a, u) \le 2^{-a}} \subseteq B_{\pair{i,a-1}}$. We need to enumerate a strong base of $\mathcal{N}(y)$. To this end we will enumerate all pairs $\pair{i,a-1}$ with $a > 0$ and $y_a \in B_{\pair{i,a}}$. Let $m$ be an index of $(y_a)_a$. By  definition of $\bar{x}$, $\set{B_d}{d \in W_{\varphi_m(a)}}$ is a strong basis of $\mathcal{N}(\bar{x}_{\varphi_m(a)})$. Therefore, 
\begin{equation}\label{eq-lim}
\bar{x}_{\varphi_m(a)} \in B_{\pair{i,a}} \Longleftrightarrow (\exists d \in W_{\varphi_m(a)})\, d \prec_B \pair{i,a}.
\end{equation}
Hence, if we let $\li \in R^{(1)}$ such that
\[
W_{li(m)} = \set{\pair{i,a-1}}{a > 0 \wedge (\exists d \in W_{\varphi_m(a)})\, d \prec_B \pair{i,a}},
\]
then all basic open set $B_{\pair{i,a-1}}$ with $\pair{i,a-1} \in W_{\li(m)}$ contain the limit point $y$. It remains to show that they form a strong filter base.

Let $\pair{i,a-1}, \pair{j, c-1} \in W_{\li(m)}$.  Then $y \in B_{\pair{i,a-1}} \cap B_{\pair{j,c-1}}$. Since the set of all $B_d$ forms a strong basis of the metric topology, there exist $b,n \in \omega$ such that $y \in B_{\pair{b,n}}$ and $\pair{b,n} \prec_B \pair{i,a-1}$ as well as $\pair{b,n} \prec_B \pair{j,c-1}$. Let $\bar{n} \in \omega$ with $2^{-\bar{n}} < 2^{-n} - \delta(\beta_b,y)$. Moreover, choose $\hat{n} > \bar{n} + 2$ so that $\delta(y_{\hat{n}}, y) \le 2^{-\hat{n}}$ and $e \in \omega$ with $\delta(\beta_e,y_{\hat{n}}) < 2^{-\hat{n}}$ as well. It then follows that
\begin{align*}
\delta(\beta_e,\beta_b) + 2^{-\hat{n}+1} 
&\le \delta(\beta_e,y) + \delta(y,\beta_b) + 2^{-\hat{n}+1}\\
&< 2 \cdot 2^{-\hat{n}+1} + \delta(y,\beta_b)
< 2^{-\bar{n}} + \delta(y,\beta_b)
< 2^{-n},
\end{align*}
which means that $\pair{e,\hat{n}-1} \prec_B \pair{b,n}$. Thus we have that $\pair{e,\hat{n}-1} \prec_B \pair{i,a-1}$ as well as $\pair{e,\hat{n}-1} \prec_B \pair{j, c-1}$. Moreover, as $\bar{x}_{\varphi_m(\hat{n})} \in B_{\pair{e, \hat{n}}} \subseteq B_{\pair{e, \hat{n}-1}}$, we obtain with (\ref{eq-lim}) that $\pair{e, \hat{n}-1} \in W_{\li(m)}$.
\end{proof}

So far we have seen for two important and large classes of effective spaces that acceptable numberings in addition have a limit algorithm. We will now, conversely, study when numberings that have a limit algorithm also allow effective limit passing. 

\begin{proposition}
Let  either $\mathcal{T}$ satisfy  $T_1$ separation or every sequence in \textsc{Seq} allow delaying. Moreover, let $x$ be computable as well as have a limit algorithm. Then $x$ also allows effective limit passing.
\end{proposition}
\begin{proof}
Let $\li \in P^{(1)}$ witness that $x$ has a limit algorithm, and $p \in R^{(1)}$ be as in Property~\ref{prop-seq-2} of \textsc{Seq}. If $m$ is an index of a normed computable enumeration of basic open sets converging to $y \in T$, then $ y = \max_{\le_\tau} \Lim_a x_{\varphi_{p(m)}(a)}$, as we have already seen. Thus, $x_{\li(p(m))} \le_\tau y$. 

If $\mathcal{T}$ satisfies $T_1$ separation, the specialization order coincides with the identity on $T$. Hence, $y =  x_{\li(p(m))}$ in this case. In the other case, every sequence in \textsc{Seq} allows delaying. As a consequence of Corollary~\ref{cor-liprop}(\ref{cor-liprop-2}) we therefore again have that  $y = x_{\li(p(m))}$. So, $\pt = \li \circ p$ witnesses that $x$ allows effective limit passing.
\end{proof}

\section{Continuity}\label{cont}

By definition, a sequence $(y_a)_a$ converges to a point $y$, if for any $n \in \omega$ with $y \in B_n$ there is some $N^y_n \in \omega$ with $y_a \in B_n$, for all $a \ge N^y_n$. If $y' \le_\tau y$, then $(y_a)_a$ converges to $y'$ as well and we can take $N^{y'}_n = N^y_n$. Hence, if $(y_a)_a$ has a maximal limit point $z$, we can choose $N^y_n = N^z_n$, for all $y \in \Lim_a y_a$.  We call any function that maps $n$ with $B_n \cap \Lim_a y_a \not= \emptyset$ to some $N_n$ with $y_c \in B_n$, for all $c \ge N_n$, a \emph{uniform convergence module} of $(y_a)_a$.

\begin{definition}
A sequence $(y_a)_a$ of elements of $T$ \emph{converges computably}, if $\Lim_a y_a \not= \emptyset$ and $(y_a)_a$ has a computable uniform convergence module, i.e., there is some function $k \in P^{(1)}$ such that for all $n \in\omega$ with $B_n \cap \Lim_a y_a \not= \emptyset$ it follows that $k(n)\conv$ and $y_c \in B_n$, for all $c \ge k(n)$.
\end{definition}

Let $m$ be an index of a computable normed enumeration of basic open sets converging to $y \in B_n$. Then, for all $a > 0$, $x_{\varphi_{p(m)}(a)} \in B_{\varphi_m(a-1)}$, where $p \in R^{(1)}$ is as in Property~\ref{prop-seq-2} of \textsc{Seq}. Assume that $\prec_B$ is computably enumerable and set 
\[
A = \set{\pair{a', m', n'}}{a' > 0  \wedge \varphi_{m'}(a'-1) \prec_B {n'}}.
\]
As $\set{B_a}{a \in \range(\varphi_m)}$ is a strong basis of $\mathcal{N}(y)$, $A$ is not empty. With respect to some fixed enumeration, let $\pair{\bar{a}, \overline{m}, \bar{n}}$ be the first element enumerated in $A$ with $\overline{m} = m$ and $\bar{n} = n$. Set $\varphi_{k(m)}(n) = \bar{a}$. Then $\varphi_{k(m)}$ witnesses that $(x_{\varphi_{p(m)}(a)})_a$ converges computably to $y$, uniformly in $m$.

\begin{lemma}\label{lm-econvcnseq}
Let $\prec_B$ be computably enumerable, $p \in R^{(1)}$ as in Property~\ref{prop-seq-2} of \textsc{Seq}, and $m$ an index of a converging normed computable enumeration of basic open sets. Then the associated sequence $(x_{\varphi_{p(m)}(a)})_a$ of points converges computably, uniformly in $m$.
\end{lemma}

We say that $\mathcal{T}$ has a \emph{uniformly computable convergence module}, if there is a function $\cm \in R^{(1)}$ such that for all indices $m$ of computably converging sequences in \textsc{Seq}, $\varphi_{\cm(m)}$ is a corresponding uniform convergence module.

Classically, a map $\fun{F}{T}{T'}$ is sequentially continuous, if for all sequences $(y_a)_a$ in $T$, $F(\Lim_a y_a) \subseteq \Lim_a F(y_a)$. It follows that $F(\max_{\le_\tau} \Lim_a y_a) \le_{\tau'} \max_{\le_{\tau'}} \Lim_a F(y_a)$, if both maxima exist. If $F$ is monotone with respect to the specialization order and $\max_{\le_\tau} \Lim_a y_a$ is an upper bound of all sequence elements $y_a$, also the converse inequality holds.

In the effective version of sequential continuity we distinguish whether when computing an approximation of $F(y)$, for some $y \in \Lim_a y_a$, the existence of a computable uniform convergence module for $(y_a)_a$ is used or not.

\begin{definition}\label{df-effseqcon}
A map $\fun{F}{T}{T'}$ is 
\begin{enumerate}
\item \emph{strongly effectively sequentially continuous}, if for every converging sequence $(y_a)_a$ in \textsc{Seq}, the sequence $(F(y_a))_a$ converges computably with $\Lim_a F(y_a) \supseteq F(\Lim_a y_a)$, uniformly in any index $m$ of $(y_a)_a$, i.e., there is some function $k' \in R^{(1)}$ so that $\lambda n.\, \varphi_{k'(m)}(n)$ witnesses the computable convergence of $(F(y_a))_a$. 

\item \emph{effectively sequentially continuous}, if for every computably converging sequence $(y_a)_a$ in \textsc{Seq} with computable uniform convergence module $k \in P^{(1)}$, the sequence $(F(y_a))_a$ converges computably with $\Lim_a F(y_a) \supseteq F(\Lim_a y_a)$, uniformly in any index $m$ of $(y_a)_a$ and any G\"odel number of $k$.
\end{enumerate}
\end{definition}  

Obviously, every strongly effectively sequentially continuous map is effectively sequentially continuous. As for computably indexed effectively pointed spaces as well as for constructive metric spaces converging sequences in \textsc{Seq} always converge computably with a uniformly computable convergence module, both notions coincide in these cases.

For topological spaces with countable topological basis it is well known that sequentially continuous maps are continuous, and vice versa. In this section we will study this relationship in the effective context described so far.

\begin{definition}\label{df-effcon}
 A map $\fun{F}{T}{T'}$  is said to be
\begin{enumerate}
\item\label{df-effcon-1} \emph{effectively pointwise continuous}, if there is a function $h \in P^{(2)}$ such that for all $i \in \dom(x)$ and $n \in \omega$ with $F(x_i) \in B'_n$, $h(i,n )\conv$, $x_i \in B_{h(i,n)}$, and $F(B_{h(i,n)}) \subseteq B'_n$;
\item\label{df-effcon-2} \emph{effectively continuous}, if there is a function $g \in  R^{(1)}$ such that for all $n \in \omega$,  $F^{-1}(B'_n)= \bigcup\set{B_a}{a \in W_{g(n)}}$.
\end{enumerate}
\end{definition}  

The interrelationship between effective and effective pointwise continuity was investigated in (Spreen 1998, 2000). 

\begin{proposition}\label{prop-cont-ptcont}
Let $x$ be computable. Then every effectively continuous mapping $\fun{F}{T}{T'}$ is effectively pointwise continuous.
\end{proposition}

For the converse implication stronger requirements are needed. Among others, strong inclusion must be computably enumerable\footnote{This condition was not mentioned when stating the result in (Spreen 1998, 2000), but used in the proof.} and $\mathcal{T}$ \emph{recursively separable}, which means that $T$ has to contain an enumerable dense subset. 

\begin{proposition}\label{prop-ptcont-cont}
Let $\mathcal{T}$ be effective and recursively separable such that $\prec_B$ is computably enumerable. Moreover, let $x$ be acceptable and $x'$ computable. Then every effective and effectively pointwise continuous map $\fun{F}{T}{T'}$ is also effectively continuous.
\end{proposition}

\begin{proposition}\label{prop-cont-eff}
Let $\mathcal{T'}$ be effective such that $\prec_{B'}$ is computably enumerable. Moreover, let $x$ be computable and $x'$ allow effective limit passing. Then every effectively continuous map $\fun{F}{T}{T'}$ is effective.
\end{proposition}

\begin{corollary}\label{cor-effcont}
Let $\mathcal{T}$ and $\mathcal{T'}$ be effective with $\prec_B$ and $\prec_{B'}$ being computably enumerable. Moreover, let $x$ and $x'$ be acceptable, and $\mathcal{T}$ recursively enumerable. Then, for every map $\fun{F}{T}{T'}$, $F$ is effectively continuous if, and only if, it is effectively pointwise continuous and effective.
\end{corollary}

Let us first see how the continuity of sequentially continuous maps is usually shown, in a non-effective context. Assume to this end that $\mathcal{T} = (T, \tau)$ and $\mathcal{T'} = (T', \tau')$, respectively, are $T_0$ spaces with countable bases $\mathcal{B}$ and $\mathcal{B'}$, let $\fun{F}{T}{T'}$ be a sequentially continuous map, and $y$ a point in $T$.
\begin{itemize}
\item First, one uses the countability of $\mathcal{B}$ to construct a sequence of basic open sets $U_0 \supseteq U_1 \supseteq \dotsb\,$ that forms a basis of the neighbourhood filter of $y$.

\item Next, one assumes that $F$ is not continuous. Hence, there is some basic open set $V$ containing $F(y)$ such that $F(U_a) \not\subseteq V$, for all $a \in \omega$. 

\item Using the Axiom of Choice, one then selects some point $y_a \in U_a$, for each $a \in \omega$, such that $F(y_a) \not\in V$. 

\item It follows that $(y_a)_a$ converges to $y$, but $(F(y_a))_a$ does not converge to $F(y)$, a contradiction.
\end{itemize}

Now assume that $B$ and $B'$ are total numberings of $\mathcal{B}$ and $\mathcal{B'}$, respectively, and that $T$ and $T'$ are both countable with numberings $x$ and $x'$, respectively. Moreover, suppose that $\mathcal{T}$ is effective and $x$ is computable. Then, by Lemma~\ref{lm-normseq}, there is some function $q \in R^{(1)}$ such that, for each $i \in \dom(x)$, $q(i)$ is an index of a normed computable enumeration of basic open set converging to $x_i$. In particular, we have that $B_{\varphi_{q(i)}(0)} \supseteq B_{\varphi_{q(i)}(1)} \supseteq \dotsb\,$.

In the second step, assuming that 
\begin{equation*}\label{eq-ninc}
B_{\varphi_{q(i)}(a)} \not\subseteq F^{-1}(B'_n)
\end{equation*}
we need to be able to effectively find a witness $y_a$ for this, uniformly in $n$, $a$ and $i$. It is here where we need $F$ to have a(n appropriate) witness for noninclusion. 
As we have seen in Corollary~\ref{cor-pt-wit}, effective maps $\fun{F}{T}{T'}$ have an appropriate witness for noninclusion, if $\mathcal{T}$ is effectively pointed. In \cite[Theorem 4.1]{sy84} an analogous result was shown for the case that $\mathcal{T}$ is recursively separable and $\mathcal{T'}$  a constructive metric space. In addition,  $x$ must be computable and have a limit algorithm, and $x'$ must be \emph{co-computable}, which means that for every $n \in \omega$, $\ext(B'_n)$, the exterior of $B'_n$, is completely enumerable, uniformly in $n$. However, as follows from an example in \cite{fr58}, this is not the case, in general.

\begin{theorem}\label{th-seqccont}
Let $\mathcal{T}$ be effective and $x$  computable. Then every strongly effectively sequentially continuous map $\fun{F}{T}{T'}$ that has an appropriate witness for noninclusion is effectively pointwise continuous.
\end{theorem}
\begin{proof}
Let $k' \in R^{(1)}$ witness the effective sequential continuity of $F$ and  $s \in P^{(2)}$ as well as $r \in P^{(3)}$ its having an appropriate witness for noninclusion. Moreover, let $q \in R^{(1)}$ be as in Lemma~\ref{lm-normseq} and $p \in R^{(1)}$ as in Property~\ref{prop-seq-2} of \textsc{Seq}. By the recursion theorem there is a function $g \in R^{(2)}$ with
\[
\varphi_{g(i,n)}(a) = \begin{cases}
                                   \varphi_{p(q(i))}(a) & \hspace{-3em}\text{ if $\varphi_{k'(g(i,n))}(s(i,n))\nconv_a$, or $\varphi_{k'(g(i,n))}(s(i,n))\conv_a$}\\
                                                                 & \text{ and $a < \varphi_{k'(g(i,n))}(s(i,n))$,} \\
                                   r(i, \varphi_{q(i)}(\varphi_{k'(g(i,n))}(s(i,n))),n) & \text{ otherwise.}
                               \end{cases}
\]
Set $h(i,n) = \varphi_{k'(g(i,n))}(s(i,n))$ and assume that $h(i,n)\nconv$, for some $i \in \dom(x)$ and $n \in \omega$ with $F(x_i) \in B'_n$. Then $s(i,n)\conv$ and $F(x_i) \in B'_{s(i,n)} \subseteq B'_n$. Moreover, $x_{\varphi_{g(i,n)}(a)} = x_{\varphi_{p(q(i))}(a)}$, for all $a \in \omega$. Thus, $(x_{\varphi_{g(i,n)}(a)})_a$ is a sequence in \textsc{Seq} converging to $x_i$. It follows that $\varphi_{k'(g(i,n))}(s(i,n))$ is defined, a contradiction. Therefore, $h(i,n)$ is defined, for all $i \in \dom(x)$ and $n \in \omega$ with $F(x_i) \in B'_n$. 

Next, assume that $B_{\varphi_{q(i)}(h(i,n))} \not\subseteq F^{-1}(B'_n)$, for some $i \in \dom(x)$ and $n \in \omega$ with $F(x_i) \in B'_n$. Then $r(i,\varphi_{q(i)}(h(i,n)),n)\conv \in \dom(x)$ with 
\[
x_{r(i,\varphi_{q(i)}(h(i,n)),n)} \in \hl(B_{\varphi_{q(i)}(h(i,n))}) \setminus F^{-1}(B'_{s(i,n)}).
\]
As the noninclusion witness is appropriate, $(x_{\varphi_{g(i,n)}(a)})_a$ is in \textsc{Seq}. Moreover, it converges to $x_{r(i,\varphi_{q(i)}(h(i,n)),n)}$.

Since $\lambda b.\, \varphi_{k'(g(i,n))}(b)$ is a module of the convergence of $(F(x_{\varphi_{g(i,n)}(a)}))_a$ to 
\[
F(x_{r(i,\varphi_{q(i)}(h(i,n)),n)}),
\]
it follows that for all $c \ge \varphi_{k'(g(i,n))}(s(i,n))$, $F(x_{\varphi_{g(i,n)}(c)}) \in B'_{s(i,n)}$. In particular, we have that 
$F(x_{\varphi_{g(i,n)}(\varphi_{k'(g(i,n))}(s(i,n)))}) \in B'_{s(i,n)}$. As 
\[
\varphi_{g(i,n)}(\varphi_{k'(g(i,n))}(s(i,n))) = r(i,\varphi_{q(i)}(\varphi_{k'(g(i,n))}(s(i,n))),n) = r(i,\varphi_{q(i)}(h(i,n)),n),
\]
it follows that $F(x_{r(i,\varphi_{q(i)}(h(i,n)),n)}) \in B'_{s(i,n)}$, a contradiction. Consequently, the function $\lambda i, n.\, \varphi_{q(i)}(h(i,n))$ witnesses that $F$ is effectively pointwise continuous. 
\end{proof}

Classically, continuous maps are also sequentially continuous. Let us see next in as much this holds in the effective setting. 
 
\begin{proposition}\label{prop-contseq}
Let \textsc{Seq} have maximal limits and $x$ a limit algorithm that computes maximal limits.  Then every effectively pointwise continuous maps $\fun{F}{T}{T'}$ is effectively sequentially continuous.
\end{proposition}
\begin{proof}
Let $h \in P^{(2)}$ witness that $F$ is effectively pointwise continuous, and assume that $m$ is an index of a computable sequence in \textsc{Seq} with computable convergence module $k \in P^{(2)}$ converging to a point $y \in T$. Moreover, let $n \in \omega$ so that $F(y) \in B'_n$. Then $(x_{\varphi_m(a)})_a$ also converges computably to its largest limit point $x_{\li(m)}$. As is easily verified, effectively pointwise continuous maps are monotone with respect to the specialization order. Therefore, $F(x_{\li(m)}) \in B'_n$. It follows that $h(\li(m),n)\conv$, $x_{\li(m)} \in B_{h(\li(m),n)}$, and $F(B_{h(\li(m),n)}) \subseteq B'_n$. Hence, $x_{\varphi_m(a)} \in B_{h(\li(m),n)}$ and consequently $F(x_{\varphi_m(a)}) \in B'_n$, for all $a \ge k(h(\li(m),n))$, which shows that $F$ is effectively sequentially continuous. 
\end{proof}

\section{Effective maps}\label{map}

In this section we will investigate when an effective map $\fun{F}{T}{T'}$ is effectively sequentially continuous.

Assume  that $(y_a)_a$ is a computable sequence in \textsc{Seq} converging to $y \in T$. Then $(F(y_a))_a$ is computable as well. We will now show that in certain general cases it converges computably to $F(y)$, i.e., we will show that in these cases each effective map is effectively sequentially continuous.

Let $m$ be an index of $(y_a)_a$ and $n \in \omega$ with $F(y) \in B'_n$. Then, uniformly in $m$ and $n$, we recursively construct a  computable sequence $(z_a)_a$ with index $b$:
\begin{enumerate}

\item We follow the sequence $(y_a)_a$ as long as the computation of $\li(b)$ has \emph{not} terminated or $F(x_{\li(b)})$ has \emph{not} been found in $B'_n$.

\item If the computation of $\li(b)$ has terminated and $F(x_{\li(b)})$ has been found in $B'_n$, say in step $N_0$,  we delay our strategy to follow the sequence $(y_a)_a$ and repeat the element $y_{N_0}$ as long as $F(y_{N_0})$ has \emph{not} been found in $B'_n$.

\item If, in step $N_1$, we have found $F(y_{N_0})$ in $B'_n$, we go to element $y_{N_0 + 1}$ and repeat it as long as we have \emph{not} found $F(y_{N_0+1})$ in $B'_n$.

\item If, in step $N_2$, we have found $F(y_{N_0+1})$ in $B'_n$, we go to element $y_{N_0 + 2}$ and repeat it as long as we have \emph{not} found $F(y_{N_0+2})$ in $B'_n$, and so on.

\end{enumerate}

As we will see, all steps $N_0, N_1, N_2. \ldots$ exist and depend computably on $m, n$. Obviously, then $N_0$ is a module for the convergence of $(F(y_a))_a$ to $F(y)$.

\begin{theorem}\label{thm-seqcont}
Let $x$ be computable and  have a limit algorithm, $x'$ be computable, and the sequences in \textsc{Seq} allow delaying. Then every effective map $\fun{F}{T}{T'}$ is strongly effectively sequentially continuous.
\end{theorem}
\begin{proof}
Since $x'$ is computable, $B'_n$ is completely enumerable, uniformly in $n$. Hence, as $F$ is effective, $F^{-1}(B'_n)$ is completely enumerable as well. Let this be witnessed by $W_{v(n)}$ with $v \in R^{(1)}$. Obviously, $v$ uniformly depends on the index of $F$. Finally, let $\li \in P^{(1)}$ witness that $x$ has a limit algorithm and let $t \in R^{(1)}$ with $\range(t) \subseteq \dom(x)$ so that $(x_{t(a)})_a$ is a sequence in \textsc{Seq} converging to some $y \in T$.

Set $\bar{g}(\bar{b},n) = \mu c: \varphi_{v(n)}(\li(\bar{b}))\conv_c$ and define $\bar{u} \in P^{(3)}$ by
\begin{align*}
&\bar{u}(\bar{b},n,0) = 0,\\
&\bar{u}(\bar{b}, n, a+1) = \begin{cases} 
                                                 \bar{u}(\bar{b},n,a)+1 & \text { if $\varphi_{v(n)}(\li(\bar{b}))\nconv_{a+1}$, or  $\varphi_{v(n)}(\li(\bar{b}))\conv_{a+1}$}\\
                                                                                      &\text{ and $\varphi_{v(n)}(t(\bar{u}(\bar{b},n,a)))\conv_{a+1}$,}\\
                                                 \bar{u}(\bar{b},n,a) & \text{ if $\varphi_{v(n)}(\li(\bar{b}))\conv_{a+1}$ and  $\varphi_{v(n)}(t(\bar{u}(\bar{b},n,a)))\nconv_{a+1}$.}
                                              \end{cases}
\end{align*}
In addition, let $h \in R^{(2)}$ with
\[
\varphi_{h(\bar{b},n)}(a) = t(\bar{u}(\bar{b},n,a)).
\]

By the recursion theorem there is then a function $b\in R^{(1)}$ with $\varphi_{b(n)} = \varphi_{h(b(n),n)}$. Set $g(n) = \bar{g}(b(n),n)$ and $u(n,a) = \bar{u}(b(n),n,a)$.

Suppose that $g(n)\nconv$, for some $n \in \omega$ with $F(y) \in B'_n$. Then $u(n,a) = a$ and hence $x_{\varphi_{b(n)}(a)} = x_{t(a)}$. It follows that $(x_{\varphi_{b(n)}(a)})_a$ converges to $y$. Hence, $y \le_\tau x_{\li(b(n))}$ because of Corollary~\ref{cor-liprop}. By assumption, $y \in F^{-1}(B'_n)$. With Corollary~\ref{cor-cecl} we therefore obtain that $x_{\li(b(n))} \in F^{-1}(B'_n)$ as well, since $F^{-1}(B'_n)$ is completely enumerable. Consequently, $\li(b(n)) \in W_{v(n)}$, i.e., $g(n)$ is defined, a contradiction.

Next, let 
\begin{align*}
&k(n,0) = \mu c: \varphi_{v(n)}(\li(b(n)))\conv_c \wedge c > g(n) \wedge \varphi_{v(n)}(t(g(n)+0))\conv_c, \\
&k(n, e+1) = \mu c: k(n,e)\conv_c \wedge c > k(n,e) \wedge \varphi_{v(n)}(t(g(n)+e+1))\conv_c,
\end{align*}
and assume that there is some $ n, \bar{a} \in \omega$ so that $k(n, \bar{a})\nconv$. Let $\bar{a}$ be minimal with this property. Then 
\[
\varphi_{b(n)}(a) = \begin{cases}
                                      t(a) & \text{ if $a < g(n)$,}\\
                                      t(g(n)) & \text{ if $g(n) \le a < k(n,0))$,}\\
                                      t(g(n)+1)) & \text{ if $k(n,0) \le a < k(n,1)$,}\\
                                      \hfill       \vdots  \hfill\\
                                      t(g(n)+\bar{a}-1) & \text{ if $k(n,\bar{a}-2) \le a < k(n,\bar{a}-1)$,}\\
                                      t(g(n)+\bar{a}) & \text{ if $ a \ge  k(n, \bar{a}-1)$.}
                                 \end{cases}
\]
                   
As \textsc{Seq} is closed under delaying, the sequence $(x_{\varphi_{b(n)}(a)})_a$ is in \textsc{Seq}. Moreover, it is a sequence that is eventually constant. With Condition~\ref{df-la}(\ref{df-la-2}) it therefore follows that $x_{\li(b(n))} = x_{t(g(n)+\bar{a})}$. As we have already seen, $\li(b(n)) \in W_{v(n)}$, i.e., $x_{\li(b(n))} \in F^{-1}(B'_n)$. The latter set is completely enumerable. Therefore, also $t(g(n)+\bar{a}) \in W_{v(n)}$, which means that $k(n,\bar{a})\conv$, a contradiction.

This shows that $k(n,e)\conv$, for all $n, e \in \omega$. In other words, $g$ is a computable convergence module, i.e., the sequence $(F(x_{t(a)})_a$ converges computably to $F(y)$. As follows from the construction, $g$ depends computably on the G\"odel number of $t$. So, we have that $F$ is sequentially continuous, strongly effectively.
\end{proof}

Note that the construction of the convergence module $g$ not only depends computably on the index of the sequence transformed by $F$, but also on the index of $F$. Moreover, we do not know whether effective maps have a witness for noninclusion under the assumptions of the theorem.

\begin{corollary}
Let  $x$ have a limit algorithm and be computable. Moreover, let the sequences in \textsc{Seq} allow delaying. Then every convergent sequence in \textsc{Seq} converges computably and $\mathcal{T}$ has a uniformly computable convergence module.
\end{corollary}
\begin{proof}
Let $F$ be the identity on $T$ in the above theorem.
\end{proof}

The  most restrictive assumption in the above result is that sequences in \textsc{Seq} should allow delaying. In the  proof the constructed sequence had to be delayed several times. This is certainly not possible for sequences that have to satisfy strong conditions as the regular Cauchy sequences. In what follows we will derive an analogous result for spaces like constructive metric spaces in which the requirement that sequences in \textsc{Seq} allow delaying is no longer used. As we will see, the construction in the proof is very much the same as the one in the previous proof, only where we had to wait and see whether a certain computation will terminate, we can now use a decision precedure. We will derive the result for the rather general class of effective $T_3$ spaces. For a subset $X$ of a topological space, let $\cl(X)$ denote its closure.

\begin{definition}\label{df-t3}
$\mathcal{T}$ is \emph{effectively} $T_3$, if there is some function $s \in P^{(2)}$ such that  $s(i,m)\conv$ with 
\[
x_i \in B_{s(i,m)} \subseteq \cl(B_{s(i,m)}) \subseteq B_m,
\]
for all $i \in \dom(x)$ and $m \in \omega$ with $x_i \in B_m$.
\end{definition}

As is shown in \cite[Lemma 3.3]{sp98}, every constructive metric space is effectively $T_3$.

\begin{theorem}\label{thm-seqcontt3}
Let $\mathcal{T'}$ be effectively $T_3$,  $x'$ be computable as well as co-computable, and $x$ have a limit algorithm. Then every effective map $\fun{F}{T}{T'}$ is strongly effectively sequentially continuous.
\end{theorem}
\begin{proof}
Let $s \in P^{(2)}$ and  $f \in P^{(1)}$, respectively, witness that $\mathcal{T}$ is effectively $T_3$ and $F$ is effective. Since $x'$ is computable, it follows as in the proof of Theorem~\ref{thm-seqcont} that $F^{-1}(B'_n$) is completely enumerable, uniformly in $n$. Let this be witnessed by $W_{v(n)}$ with $v \in R^{(1)}$.

As $x'$ is also co-computable, it follows in a similar way that there is a function $w \in R^{(2)}$ so that $W_{w(i,n)}$ witnesses that $F^{-1}(\ext(B'_{s(f(i),n)}))$ is completely enumerable, uniformly in $i,n$.  Again $v, w$ uniformly depend on the G\"odel number of $f$.  Let, finally, $\li \in P^{(1)}$ witness that $x$ has limit algorithm and let $t \in R^{(1)}$ with $\range(t) \subseteq \dom(x)$ so that $(x_{t(a)})_a$ is a sequence in \textsc{Seq} converging to some point $y \in T$. 

In the construction of the sequence with index $b$ we are going to describe, for certain sequence elements $x_{t(a)}$ we will search whether we find $F(x_{t(a)}) \in B'_n$ or $F(x_{t(a)}) \in \ext(B'_{s(f(\li(b)),n)})$. Possibly both is the case, then we give preference to what we find first. 

The definitions of the functions that we introduce next are parts of a large mutual recursion.
Let 
\[
\bar{k}_1(\bar{b},n,a) = \mu c: \varphi_{v(n)}(t(\bar{u}(\bar{b},n,a)))\conv_c, \quad
\bar{k}_2(\bar{b},n,a) = \mu c: \varphi_{w(\li(\bar{b}),n)}(t(\bar{u}(\bar{b},n,a)))\conv_c,
\]
and
\[
\bar{k}(\bar{b},n,a) = \begin{cases}
                                               \bar{k}_1(\bar{b},n,a) & \text{ if $\varphi_{v(n)}(t(\bar{u}(\bar{b},n,a)))$ terminates in at most the same}\\
                                                                                      & \text{ number of steps as $\varphi_{w(\li(\bar{b}),n)}(t(\bar{u}(\bar{b},n,a)))$,}\\
                                               \bar{k}_2(\bar{b},n,a) & \text{ if $\varphi_{w(\li(\bar{b}),n)}(t(\bar{u}(\bar{b},n,a)))$ terminates in less steps than}\\ 											& \text{ $\varphi_{v(n)}(t(\bar{u}(\bar{b},n,a)))$,}\\
                                              \text{undefined} & \text{ otherwise}.
                                        \end{cases} 
\]
Moreover, set $\bar{g}(\bar{b},n) = \mu c: \varphi_{v(n)}(\li(\bar{b}))\conv_c$ and define $\bar{u} \in P^{(3)}$ by
\begin{align*}
&\bar{u}(\bar{b},n,0) = 0,\\
&\bar{u}(\bar{b},n,a+1) = \begin{cases}
                                              \bar{u}(\bar{b},n,a)+1 & \text{ if $\varphi_{v(n)}(\li(\bar{b}))\nconv_{a+1}$, or $\varphi_{v(n)}(\li(\bar{b}))\conv_{a+1}$, $\bar{k}(\bar{b},n,a)\conv$,}\\
                                                                                     & \text{ $\bar{k}_1(\bar{b},n,a)\conv$ and $\bar{k}(\bar{b},n,a) = \bar{k}_1(\bar{b},n,a)$,}\\
                                             \bar{u}(\bar{b},n,a) & \text{ if  $\varphi_{v(n)}(\li(\bar{b}))\conv_{a+1}$,  $\bar{k}(\bar{b},n,a)\conv$, $\bar{k}_2(\bar{b},n,a)\conv$}\\
                                             & \text{ and $\bar{k}(\bar{b},n,a) = \bar{k}_2(\bar{b},n,a)$,}\\
                                             \text{undefined} & \text{ otherwise.}
                                          \end{cases}
\end{align*}
Finally, let $h \in R^{(2)}$ with
\[
\varphi_{h(\bar{b},n)}(a) = t(\bar{u}(\bar{b},n,a)).
\]

By the recursion theorem there is then some function $b\in R^{(1)}$ with $\varphi_{b(n)} = \varphi_{h(b(n),n)}$. Set
\begin{gather*}
k_1(n,a) = \bar{k}_1(b(n),n,a),  \quad k_2(n,a) = \bar{k}_2(b(n),n,a), \quad k(n,a) = \bar{k}(b(n),n,a),\\
g(n) = \bar{g}(b(n),n) \quad \text{and} \quad u(n,a) = \bar{u}(b(n),n,a).
\end{gather*}

Then it follows as in the proof of Theorem~\ref{thm-seqcont} that $g(n)$ is defined for all $n \in \omega$ with $F(y) \in B'_n$. Thus, $F(x_{\li(b(n))}) \in B'_n$, i.e., $x'_{f(\li(b(n)))} \in B'_n$, from which we obtain with the effective $T_3$ property that  $x'_{f(\li(b(n)))} \in B'_{s(f(\li(b(n))),n)}$. In addition, we have for $a \ge g(n)$ that either $F(x_{t(u(n,a))}) \in B'_n$, or $F(x_{t(u(n,a))}) \in T' \setminus B'_n$, in which case $F(x_{t(u(n,a))}) \in \ext(B_ {s(f(\li(b(n))),n)})$. Therefore, at least one of $k_1(n,a)$ and $k_2(n,a)$ must be defined, i.e., $k(n,a)$ is always defined. By induction on $a$ it follows that $u(n,a)$ is always defined as well.

Assume that for some $n, \bar{a} \in \omega$ with $F(y) \in B'_n$ and $\bar{a} \ge g(n)$, $k_2(n, \bar{a})$ is defined and $k(n, \bar{a}) =  k_2(n, \bar{a})$, and let $\bar{a}$ be minimal with this property. Then not the first but the second case in the definition of $u(n,a)$ comes into action. And once this case is active, it will remain active for ever, implying
\[
\varphi_{b(n)}(a) = \begin{cases}
                                     t(a) & \text{ if $a < \bar{a}$,}\\
                                     t(\bar{a}) & \text{ otherwise.}
                             \end{cases}
\]
With Property~\ref{prop-seq-3} of \textsc{Seq} it follows that $(x_{\varphi_{b(n)}(a)})_a$ is in \textsc{Seq}. Furthermore, the sequence is eventually constant. Because of Condition~\ref{df-la}(\ref{df-la-2}) we therefore have that $x_{\li(b(n))} = x_{t(\bar{a})}$. 

Note that $F(x_{\li(b(n))}) = x'_{f(\li(b(n)))}$ and $x'_{f(\li(b(n)))} \in B'_{s(f(\li(b(n))),n)}$, as we have already seen. On the other hand, as $k_2(n, \bar{a})$ is defined, we obtain that $t(\bar{a}) \in W_{w(\li(b(n)),n)}$, i.e., $F(x_{t(\bar{a})}) \in \ext(B'_{s(f(\li(b(n))),n)})$, a contradiction.

It follows for all $n \in \omega$ with $F(y) \in B'_n$ and $a \ge g(n)$ that $k(n,a)$ is defined, but $k(n,a) \not= k_2(n,a)$. Therefore, $k(n,a) = k_1(n,a)$, which implies that $k_1(n,a)$ is defined, $u(n,a) = a$ and hence $F(x_{t(a)}) \in B'_n$.
This shows that the sequence $(F(x_{t(a))})_a$ converges computably to $F(y)$.  In other words, $F$ is strongly effectively sequentially continuous.                                
\end{proof}

The two theorems in this section cover a large variety of cases. Unfortunately, however, we were not able to pursue our programme in full generality, i.e., to derive a theorem stating the effective sequential continuity of effective operators that would include all interesting cases, as we did in the continuous case. However, the present situation is also more complicated. In the continuous case we have one decision to make whether to follow a given sequence or to deviate. Now, we have to deal with infinitely many such decisions and the strategies how to make the decisions were quite different in the cases we considered. It is even not clear to us whether effective operators are effectively sequentially continuous in general, or whether an additional condition is needed.

\begin{proposition}\label{fried}
There is a constructive metric space $\mathcal{M}$,
a  constructive  domain $\mathcal{Q}$, and a map
$\fun{F}{M}{Q}$ which is effective, but not sequentially continuous.
\end{proposition}
\begin{proof}
The following construction is a modification of an example given by  Friedberg~(1958).
Let $\mathcal{M}$ be Baire space and $\mathcal{Q}$ Sierpinski space $\{\bot, 0\}$ with $\bot \sqsubseteq 0$, $\beta_0 = \bot$, and $\beta_{n+1} = 0$.
Moreover, set
\[
h(i) = \begin{cases}
                       1 &  \text{ if $[(\forall a \le i) \varphi_i(a) = 0] \vee (\exists c) [\varphi_i(c) \not= 0 \wedge (\forall a < c) \varphi_i(a) = 0 \wedge \mbox{}$}\\
                             &   (\exists j < c) (\forall b \le c)\varphi_i(b) = \varphi_j(b)],\\ 
                       \text{undefined} &  \text{ otherwise.}
           \end{cases} 
\]
Then $h \in P^{(1)}$. As it is readily verified, for all
$\varphi_i$, $\varphi_j \in R^{(1)}$ with $\varphi_i = \varphi_j$ one
has that $h(i) = h(j)$. Let $x$ be an admissible indexing of $Q$. Then
there is a function $d \in R^{(1)}$ such that for all $i \in \omega$
for which $\beta(W_i)$ is directed, $x_{d(i)}$ is the least upper
bound of $\beta(W_i)$. Let $q
\in R^{(1)}$ with $W_{q(i)} = \{ 0,  h(i) \}$, and set $t = d \circ q$. 
We define the effective mapping $\fun{F}{R^{(1)}}{Q}$ by $F(\varphi_i)
= x_{t(i)}$. Then $F(\varphi_i) = 0$, if the first condition  in the definition
of $h$ holds; otherwise, $F(\varphi_i) = \bot$.

Now, for $m \in \omega$, let $k_m = \max \{\, \varphi_i(m+1) + 1 \mid i \le m \wedge \varphi_i \in R^{(1)} \,\}$ and define
\[
g_m(a) = \begin{cases}
                             0 & \text{ if $ a \not= m+1$,} \\
                             k_m & \text{ otherwise.}
               \end{cases} 
\]
Then  $g_m \in R^{(1)}$, for every $m \in \omega$. Moreover, $(g_m)_m$ is a regular Cauchy sequence that converges computably to $\lambda n. 0$.  Since for any
G\"odel number $j$ of $g_m$ we have that $j >  m$ and as
$g_m(m+1) \not= 0$, it follows from the definition of $F$ that
$F(g_m) = \bot$, for all $m \in \omega$. On the other hand, $F(\lambda n. 0) = 0$.
Thus, $F$ cannot be sequentially continuous.
\end{proof}

However, this leaves still open the question whether effective operators are (strongly) effectively sequentially continuous in general. To decide this question negatively, one would need a computable sequence $(g_m)_m$ in the construction.

\paragraph{\textbf{Remark added in proof.}} After this paper was written, M.\ Hoyrup (2015) answered the above question negatively: In general, effective operators are \emph{not} effectively sequentially continuous.  Let $\fun{F}{M}{Q}$ be as in Proposition~\ref{fried} and for $m \in \omega$, set $g_m(a) = 1$, if $a = m+1$, and $g_m(a) = 0$, otherwise. Then $(g_m)_m$ is a computable regular Cauchy sequence computably converging to $\lambda n. 0$, but $(F(g_m))_m$ does not converge to $F(\lambda n. 0)$. Assume to the contrary that $0 \in \Lim_m F(g_m)$. Then there is a set $\set{j_m}{m \in \omega}$ of G\"odel numbers with finite complement so that $\varphi_{j_m}(a) = g_m(a)$, for $a \le m+1$. This is impossible as each $\varphi_{j_m}$ has infinitely many G\"odel numbers.

\section*{Acknowledgement}

The author is grateful to the anonymous referees for their careful reading of the man\-u\-script and the useful comments.

\end{document}